\RequirePackage[l2tabu, orthodox]{nag}
\documentclass[10pt, draft]{article}

\usepackage{amsmath,amssymb,amsbsy,amsfonts,amsthm,latexsym,amsopn,amstext,
                            amsxtra,euscript,amscd}
\usepackage{booktabs}
\usepackage{enumitem}
\usepackage{tikz}
\usepackage[norefs,nocites]{refcheck}

\numberwithin{equation}{section}
\newtheorem{theorem}{Theorem}[section]
\newtheorem{lemma}[theorem]{Lemma}

\newtheorem{definition}[theorem]{Definition}
\newtheorem{thm}[theorem]{Theorem}

\newtheorem{lem}[theorem]{Lemma}

 \DeclareMathOperator{\lcm}{lcm}

\def\\{\cr}
\def\({\left(}
\def\){\right)}
\def\[{\left[}
\def\]{\right]}
\def\<{\langle}
\def\>{\rangle}

\def\C{\mathbb{C}}

\def\cD{\mathcal D}

\def\cR{\mathcal R}

\def\cN{\mathcal N}

\def\R{\mathbb{R}}

\def\M{\mathbb M}

\def\F{\mathbb{F}}

\def\Z{\mathbb{Z}}

\def\notdivides{\mathrel{\kern-3pt\not\!\kern3.5pt\bigm|}}

\begin{document}


\title{\Large \textbf{Tuples of polynomials over finite fields with pairwise coprimality conditions}}
\author{%
\scshape{JUAN ARIAS DE REYNA}\\
{Department of Mathematical Analysis, Seville University} \\
{Seville, Spain}\\
\texttt{ arias@us.es}
\and
\scshape {RANDELL HEYMAN}  \\
School of Mathematics and Statistics,\\ University of New South Wales \\
Sydney, Australia\\
\texttt {randell@unsw.edu.au}
}

\maketitle


\begin{abstract}
Let $q$ be a prime power. We estimate the number of tuples of degree bounded monic polynomials $(Q_1,\ldots,Q_v) \in (\mathbb{F}_q[z])^v$ that satisfy given pairwise coprimality conditions. We show how this generalises from monic polynomials in finite fields to Dedekind domains with a finite norm.
\end{abstract}

\newline


\section{Introduction}

The question of calculating the number of relatively prime polynomials of fixed degree in finite fields arose in \cite [Section 4.6.1, Ex.5]{Knu}. Further results can be found in \cite{Cor,Rei,Arm,Ben,Hou}.

This naturally leads to the concept of tuples of polynomials in finite fields that exhibit pairwise coprimality conditions. This concept is also relevant to polynomial remainder codes used in, amongst other things, error correction (see \cite{Sto} for an early paper). The density of pairwise coprime polynomials in tuples of finite fields can be inferred from a recent paper \cite{Mic}. We improve on this result in two ways. Firstly, we contemplate generalised pairwise coprimality conditions. That is, conditions that require some, not necessarily all, of the pairs of polynomials to be coprime.
Secondly, we give an asymptotic counting formula rather than simply a density.

Our results regarding polynomials in finite fields can be applied to the more general setting of ideals in Dedekind domains. We explain this further in Section \ref{S:PS}.

Our result is heavily based on \cite{Ari}; a paper that estimates tuples of pairwise coprime integers of bounded height.  We use a graph to represent the required primality conditions as follows. Let $G=(V,E)$ be a graph with $v$ vertices and $e$ edges. The set of vertices, $V$, will be given by $V=\{1,\ldots,v\}$ whilst the set of edges of $G$, denoted by $E$, is a subset of the set of pairs of elements of $V$. That is, $E \subseteq \{\{1,2\},\{1,3\},\ldots,\{r,s\},\ldots,\{v-1,v\}\}$. We admit isolated vertices (that is, vertices that are not adjacent to any other vertex). An edge is always of the form $\{r,s\}$ with $r \ne s$ and $\{r,s\}=\{s,r\}$.
Let $$X=\{(Q_1,\ldots,Q_v) \in (\mathbb{F}_q[z])^v: Q_r \text{~monic},1 \le r\le v\}.$$
For each real $x>0$ and any prime power $q$, we define the set of all tuples that satisfy the primality conditions by
$$Y_G(x):=\{(Q_1,\ldots,Q_v)\in X:\deg Q_r\le x,~\gcd(Q_r,Q_s)=1~\text{if}~\{r,s\}\in E\}.$$
We also let $g(x)=|Y_G(x)|$, and denote with $d$ the maximum degree of the vertices
of $G$.
All references to polynomials in $\F_q[z]$ will refer to monic polynomials.
Finally, let $Q_G(z)=1+B_2z^2+\cdots + B_vz^v$ be the polynomials associated to the graph $G$, defined by
\begin{align}
Q_G(z)=\sum_{F \subset E} (-1)^{|F|}z^{|v(F)|},\qquad Q_G^+(z)=\sum_{F\subset E}z^{|v(F)|},
\end{align}
where $|v(F)|$ is the number of non-isolated vertices of graph $F$.

Our main result is as follows.
\begin{thm}
\label{pairwise ff}
For a natural number $n>1$, let $g(n)$ be the cardinality of tuples of monic polynomials
$(Q_1, \dots, Q_v)$ in $\F_q[z]$ of degree $\deg(Q_r)\le n$ satisfying the coprimality
conditions given by the graph $G$ whose vertices have degree $\le d$. Then for
any $0<\varepsilon<\frac12$ we have
\begin{equation}\label{e:main}
g(n)=\frac{\rho_{G,q}}{(q-1)^v}q^{nv}\Bigl(1+ O_{G,q}(n^d q^{-n})
+O_{G,q}(q^{-(1-\varepsilon)n})\Bigr).
\end{equation}  
where
\begin{equation}
\rho_{G,q}=\prod_{\substack{P \in \F_q[z]\\P \,\emph{irreducible}}} Q_G(q^{-\deg(P)}).
\end{equation}
\end{thm}

The dependence of the  constants in the $O$ symbols on $G$ and $q$ may be given explicitly.
Namely,
\begin{align*}
O_{G,q}(n^d q^{-n})&=\frac{\exp(d) 2^{2^e}(q-1)v}{\rho_{G,q}} O(n^d q^{-n})\\
&\text{and}\\
O_{G,q}(q^{-(1-\varepsilon)n})&=
\frac{3}{\rho_{G,q}}\Bigl(\sum_{j=1}^e \rho_{G'_j,q}^+\Bigr)O(q^{-(1-\varepsilon)n}),
\end{align*}
where the constants in the new $O$ terms  are absolute,
\begin{equation}\label{E:rho}
\rho^+_{G,q}=\prod_PQ_G^+(q^{-\deg(P)})
\end{equation}
and $G'_j$ is the graph obtained from $G$ by removing the edge $j$.

We make some comments about $\rho_{G,q}$. Letting $p(n)$ represent the total number of $v$-tuples of monic polynomials in $\mathbb{F}_q[x]$ of degree less than equal to $n$, we observe that
$$|p(n)|=\(\frac{q^n-1}{q-1}\)^v.$$
Thus, the density of v-tuple of polynomials in $\mathbb{F}_q[x]$ that are monic and have the coprimality conditions induced by graph $G$ is given by
$$\lim_{n \rightarrow \infty}\frac{g(n)}{p(n)}=\rho_{G,q}.$$

The formula for the density, $\rho_{G,q}$, is explicit. In \cite[Section 4]{Ari} we outlined the calculations for the density of 4-tuples of integers with given pairwise coprimality conditions. The calculations to obtain $\rho_{G,q}$ in the case of polynomials in finite fields can be approached in the same way. In fact the calculations to obtain $\rho_{G,q}$ in \eqref{E:rho} are easier. We can group the polynomials by degree, since each polynomial of a given degree contributes in the identical way in the product formula for $\rho_{G,q}$.

\subsection{Notation}

\begin{itemize}[label={}]
\item $\subset$ is used to indicate subset, including the equality case.
\item $\cD$ is a Dedekind domain with the finite norm property.
\item $I(\cD)$ is the set of non-null ideals in $\cD$.
\item $\cR(\cD)$ is the the ring of ideals defined in Section \ref{SS:RD}. Its elements include $Z,Z^+$ and $W$.
\item $X$, $Y$ are sets of tuples of polynomials.
\item $|X|$ is the cardinality of the set $X$.
\item $G=(V,E)$ is a graph with set of vertices $V$ and edges $E$.
\item $Q$ usually runs  over non-null ideals in $\cD$. In the case $\cD=\F_q[z]$, the ring of polynomials with coefficients in
the Galois field $\F_q[z]$, these ideals can be identified with  monic polynomials.
\item $A$, $B$, $R$, $S$ and occasionally other variables denote non-null ideals on $\cD$.
\item $P$ are prime ideals in $\cD$ or monic irreducible polynomials when $\cD=\F_q[z]$. 

\item $\omega(Q)$ denotes the number of distinct primes dividing the ideal $Q$.
Or the number of distinct monic irreducible polynomials dividing the polynomial $Q$ when
$\cD=F_Q[z]$.

\item $\mu(Q)$ denotes $0$ if the ideal $Q$ is divisible by $P^2$ (the square of some
prime ideal)  or $(-1)^{\omega(Q)}$ if $Q$ is squarefree. In the case of monic polynomials we define $\mu(Q)$ analogously.

\item $\omega(q^n)$ denotes the number of monic polynomials of degree less than or equal
to $n$ in $F_q[z]$, (see definition \ref{D:w}).

\item $\cN(Q)$ is the norm of the ideal $Q$, that is, $\cN(Q)=|\cD/Q|$. If $\cD=\F_q[z]$ then $\cN(Q)$ is also equal
to $q^{\deg Q}$.
\item $r$, $s$  are vertices of $G$, so that $1\le r, s\le v$.
\item $\{r,s\}$ is a typical edge of the graph $G$.
\item $d$ denotes the maximum degree of the vertices of $G$.
\item $e$ denotes the number of edges of $G$.
\end{itemize}

\section{Problem setting}\label{S:PS}

The general setting of our problem refers to a Dedekind domain \cite[Chap.~1]{Na}.
We recall that in a Dedekind domain, $\cD$,
every nonzero ideal $Q$ has a unique factorization
 of prime ideals $Q=P_1^{\alpha_1}\cdots P_k^{\alpha_k}$.
We also require that each ideal in $\cD$ has a finite norm. That is,  $\cN(Q):=|\cD/Q|<\infty$.
These Dedekind domains with finite norm property are considered in \cite[p.~11]{Na}.  The
norm is multiplicative $\cN(Q_1Q_2)=\cN(Q_1)\cN(Q_2)$, and for any given positive real number
$x$ the number of ideals $Q$ with $\cN(Q)\le x$ is finite.

Our main example  is the ring of polynomials with coefficients in a Galois
field $\cD=\F_q[z]$. But there is also  another important case when  $\cD$
is the ring of integers of a number field.

For any graph $G=(V,E)$ we seek an estimate of the cardinality of the set
$$G_{\cD}(x)=\{(Q_1,\dots,Q_v)\in I(\cD)^v:  \cN(Q_r)\le x, \gcd(Q_r,Q_s)=1 \text{ if }
\{r,s\}\in E\}.$$

\subsection{The ring of ideals}\label{SS:RD}
As a tool in our reasoning we consider formal sums of the type
$$\sum_{Q\in I(\cD)}\frac{n(Q)}{Q},$$
where the coefficients $n(Q)$ are real (or complex) numbers.  We define  two operations,
sum and product, by the rules
$$\Bigl(\sum_{Q\in I(\cD)}\frac{n(Q)}{Q}\Bigr)+
\Bigl(\sum_{Q\in I(\cD)}\frac{m(Q)}{Q}\Bigr)=\Bigl(\sum_{Q\in I(\cD)}\frac{n(Q)+m(Q)}{Q}\Bigr)$$
and
$$\Bigl(\sum_{Q\in I(\cD)}\frac{n(Q)}{Q}\Bigr)
\Bigl(\sum_{Q\in I(\cD)}\frac{m(Q)}{Q}\Bigr)=\sum_{Q\in I(\cD)}\frac{1}{Q}
\Bigl(\sum_{BC=Q}n(B)m(C)\Bigr).$$
Since there is only a finite set of pairs of ideals with $BC=Q$ the product is well
defined. It is clear that this makes the set of these sums, $\cR(\cD)$, a ring.

A particular element of this ring is $Z$ is given by
$$Z=\sum_{Q}\frac{1}{Q}.$$
The unique factorization of ideals in a Dedekind ring gives us
$$Z=\sum_{Q}\frac{1}{Q}=\prod_P\Bigl(1+\frac{1}{P}+\frac{1}{P^2}+\cdots\Bigr).$$
Here $P$ runs through the prime ideals. To give a meaning to the infinite product
we may consider simply the product topology in $\R^{I(\cD)}$ of the discrete topology in $\R$.

We will write
\begin{equation}\label{E:triangle}
\Bigl(\sum_{Q\in I(\cD)}\frac{n(Q)}{Q}\Bigr) \vartriangleleft
\Bigl(\sum_{Q\in I(\cD)}\frac{m(Q)}{Q}\Bigr)
\end{equation}
to denote that for any ideal $Q$ we have $n(Q)\le m(Q)$.

Observe that for any
$\sigma>0$ we obtain from \eqref{E:triangle} that, when the series converge,
$$\Bigl(\sum_{Q\in I(\cD)}\frac{n(Q)}{\cN(Q)^\sigma}\Bigr) \le
\Bigl(\sum_{Q\in I(\cD)}\frac{m(Q)}{\cN(Q)^\sigma}\Bigr).$$

We will use the notation
$$\cN^\sigma\Bigl(\sum_{Q\in I(\cD)}\frac{n(Q)}{Q}\Bigr)$$ to denote
$$\Bigl(\sum_{Q\in I(\cD)}\frac{n(Q)}{\cN(Q)^\sigma}\Bigr).$$
In particular $\cN^\sigma(Z)$ is the Dedekind zeta function in the case of a
number field. Notice that $\cN^\sigma(XY)=\cN^\sigma(X)\cN^\sigma(Y)$ for any $X$,
$Y\in\cR(\cD)$.

\subsection{The particular case of polynomials in finite fields}
Our reasoning depends heavily on the function $f(x)$ that counts the number of ideals of
$\cD$ having norm $\le x$. This function behaves very differently in the case of polynomials
and for the integers of a number field. So that in spite of our arguments being general
we consider here only the case of $\cD=\F_q[z]$.

Therefore for us $\cD=\F_q[z]$ and in this case each non-null ideal is generated by
a unique monic polynomial. Therefore we speak of ideals or monic polynomials  indistinctly and
use the letter $Q$ to denote them.  The norm of a ideal $Q$ depends only on the degree
of the corresponding polynomial. Specifically, $\cN(Q)=q^r$ if $\deg(Q)=r$. Therefore instead of
consider tuples of polynomials with $\cN(Q)\le x$ we will  consider tuples of polynomials
of degree $\deg(Q)\le n$.

In the rest of the paper we have fixed a graph $G=(V,E)$, a natural number $n$ and a Dedekind
ring $\cD=\F_q[z]$ where $q$ is a fixed prime power.  With these elements we construct
a general set of tuples
\begin{equation}\label{E:defX}
X=X(n)=\{(Q_1,\dots, Q_v)\in \F_q[z]^v\colon \deg(Q_r)\le n, 1\le r\le v\},
\end{equation}
and a set of tuples satisfying the coprimality conditions
\begin{equation}\label{E:defY}
Y=Y_G(n)=\{(Q_1,\dots, Q_v)\in X\colon \gcd(Q_r,Q_s)=1\text{ for any edge } \{r,s\}\in E\}
\end{equation}
We are interested in $g(n)=|Y_G(n)|$.

\section{Exact formula for the number of tuples\newline satisfying the coprimality conditions}

We give here a general formula to compute $g(x)$ in the case where $\cD=\F_q[z]$.
We begin with two definitions.
\begin{definition}
Given the graph $G=(V,E)$, an \emph{edge labeling} is a tuple $(Q_1,\dots,Q_e)$ of non-null
ideals in $\cD$, associating an ideal $Q_a$ to each edge $a\in E$.
Analogously we consider \emph{vertex labelings}. These are associations
$(Q_1,\dots, Q_v)$ of an ideal $Q_j$ for each vertex $j\in V$.
\end{definition}
Frequently (see, for example, the proof of Lemma \ref{P:start}) we start with an edge labeling $(N_1,\dots, N_e)$ and associate to it
a vertex labeling $(M_1,\dots, M_v)$ in the following way.
\begin{definition}\label{D:assoc}
Given an edge labeling $(N_1,\dots, N_e)$ \emph{the associated vertex labeling}
$(M_1,\dots, M_v)$ is defined by
\begin{equation}
M_r=\lcm\{N_a:\text{ the edge $a=\{r,s\}$ joins the vertex $r$ with any other one $s$}\},
\end{equation}
where $\lcm(\emptyset)=1$.
\end{definition}
We introduce some notation for the number of monic polynomials of degree $m$ or less.
\begin{definition}\label{D:w}
Let
$$
w(q^m)=
\begin{cases}
0 &{m<0}, \\
\frac{q^{m+1}-1}{q-1} & {m \ge 0}.
\end{cases}
$$
\end{definition}
Any time we use the notations $N_a$ and $M_j$ we assume that $(M_1,\dots, M_v)$ is
the associated vertex labelling to the edge labeling $(N_1,\dots, N_e)$.

The general formula for $g(x)$ will be an application of the Principle of inclusion-exclusion.

\begin{thm}[Principle of inclusion-exclusion]
Let $X$ be a finite set and let $Y_j$ be subsets of $X$ for $1\le j\le n$.
Then
$$\Bigl|X\smallsetminus \bigcup_{j=1}^n Y_j\Bigr|=\sum_{J\subset\{1,2,\dots, n\}}
(-1)^{|J|}|Y_J|,$$
where $Y_J=X$ for $J=\emptyset$ and $Y_J=\bigcap_{j\in J} Y_j$ for $J\ne\emptyset$.
\end{thm}

\begin{lem}
Consider the graph $G=(V,E)$ and the natural number $n$.

Then  $g(n)$, the number of tuples
of polynomials $(Q_1, \dots, Q_v)$ of degree $\le n$ satisfying the coprimality
conditions imposed by $G$, is given by the expression
\begin{equation}\label{E:exact}
g(n)=\sum_{\substack{(N_1,\ldots,N_e)\\\deg(N_e)\le n}}
\mu(N_1)\cdots\mu(N_e)\prod_{r=1}^v w(q^n/q^{\deg(M_r)}).
\end{equation}
\end{lem}

\begin{proof}
We  have a graph $G=(V,E)$ as described in Section \ref{S:PS}. We start
with the set of all tuples of polynomials
$X$ defined in \eqref{E:defX}
We want to define subsets $Y_j$ of $X$ so that the difference
$$X\smallsetminus \bigcup_{j=1}^n Y_j$$
is the set $Y$ defined in \eqref{E:defY}, so that $g(n)=|Y|$.

If an element $(Q_1, \dots, Q_v)\in X$ is not in $Y$ there is an edge $e=\{r,s\}$ in
$G$ such that
$\gcd(Q_r, Q_s)\ne1$. Then there is an irreducible monic polynomial $P$ such that
$P\mid Q_r$ and
$P\mid Q_s$. Obviously this polynomial $P$ will be of degree less than or equal to $n$.

Therefore we define a set $Y_{(P,e)}$ for any irreducible polynomial $P$ of degree $\le n$
and any edge $e=\{r,s\}$  of $G$ by
$$Y_{(P,e)}=\{(Q_1, \dots, Q_v)\in X: P\mid Q_r \text{ and } P\mid Q_s\}.$$
With these definitions it is clear that
$$Y=X\smallsetminus \bigcup_{P, e}Y_{(P,e)}.$$

To apply the principle of inclusion-exclusion we consider the intersection
$$Y_J=\bigcap_{j=1}^m Y_{(P_j,e_j)}, \qquad J=\{(P_1,e_1),\dots, (P_m,e_m)\}$$
of a finite set of subsets $Y_{(P,e)}$.

A tuple $(Q_1,\dots, Q_v)\in X$ is in the intersection $Y_J$ if and only if
$P_j\mid Q_{r_j}$ and $P_j\mid Q_{s_j}$ for any of the edges $e_j=\{r_j,s_j\}$.
For any edge $e$ let  $N_e=\lcm\{P\colon (P,e)\in J\}$ with $N_e=1$ if the
set $\{P\colon (P,e)\in J\}$ is empty. In this way, given $J$,  we have associated  a
monic  polynomial to any edge in $E$. With this notation it is obvious that
a tuple $(Q_1,\dots, Q_v)\in X$ is in the intersection $Y_J$ if and only if
$N_e\mid Q_r$ and $N_e\mid Q_s$ for any edge $e=\{r,s\}\in E$.
Notice that the polynomials $N_e$ associated in this way to a given $J$ are
squarefree, because they are the least common multiple of a set of
irreducible polynomials. We note that $J$ determines the monic
squarefree polynomials 
$N_1$,\dots, $N_e$; one for each edge in $E$ whose factors are all of degree less than
or equal $n$. Conversely, the monic
squarefree polynomials
$N_1$,\dots, $N_e$, one for each edge in $e$ whose factors are all of degree less than
or equal $n$, determines $J$.

Looking at it in another way the polynomial $Q_r$ associated to a given
vertex $r$ needs to be divisible by
$N_{e_1}$,\dots, $N_{e_\ell}$ if $e_j$ are the edges joining the vertex $r$ with some other (it maybe none, $\ell=0$ when the vertex is isolated).
The condition on $Q_r$ is equivalent to $M_r\mid Q_r$, where
$M_r=\lcm(N_{e_1},\dots, N_{e_\ell})$ (taking $\lcm(\emptyset)=1$, in the case where
there is no condition on $Q_r$).
Thus, associated to the given finite set $J$ of pairs $(P,e)$ we have associated
a tuple of polynomials $(M_1,\dots, M_v)$ such that
$$Y_J=\bigcap_{j=1}^m Y_{(P_j,e_j)}=
\{(Q_1,\dots, Q_v)\in X\colon M_r\mid Q_r, 1\le r\le v\}.$$
Notice that the polynomials $M_r$ associated in this way to a given $J$ are
squarefree because they are the least common multiple of a set of squarefree polynomials.

The tuple of monic polynomials $(Q_1,\dots, Q_v)$ is in $Y_J$ if each component
$Q_r$ satisfies two conditions. Firstly, $\deg(Q_r)\le n$ for $Q_r$ to be in $X$, and secondly $M_r\mid Q_r$ for $Q_r$ to be
in $Y_J$. These conditions will be satisfied by any product $M_rA$  of $M_r$
with any monic
polynomial $A$ such that $\deg(M_r)+\deg(A)\le n$. Therefore if $\deg(M_r)>n$, there
is no possible $Q_r$. When $\deg(M_r)\le n$ let $m=n-\deg(M_r)$. Then $A$ can be any
monic polynomial of degree $\le m$. The number of possible polynomials $A$,
and therefore the number
of possible $Q_r$, are
$$1+q+q^2+\cdots +q^m=\frac{q^{m+1}-1}{q-1}.$$
With these notation the cardinality of $Y_J$ can be computed as
$$|Y_J|=\prod_{r=1}^v w(q^n/q^{\deg(M_r)}),$$
where the $w$ function is as shown in Definition \ref{D:w}.
We now compute $|J|$. This is the total number of prime factor across all the $N_j$.
As mentioned before $N_j$ is squarefree, so
$$(-1)^{|J|}=(-1)^{\sum_{j=1}^e \omega(N_j)}=\mu(N_1)\cdots\mu(N_e).$$

Therefore the Principle of inclusion-exclusion yields
$$g(n)=|Y|=\sum_{N_1}\cdots\sum_{N_e}\mu(N_1)\cdots\mu(N_e)\prod_{r=1}^v w(q^n/q^{\deg(M_r)}),$$
where the summations are over all monic squarefree polynomials $N_j$
with irreducible factors of degree $\le n$. We notice that if
some $N_j$ have an irreducible factor of degree $>n$, then some $M_r$ will have
a degree $>n$ and the corresponding sum will then be zero because the factor
$w(q^n/q^{\deg(M_r)})$ will equal zero. Also, if some $N_j$ is not squarefree the factor $\mu(N_j)=0$.

Therefore the sum with the restricted conditions on $N_j$ will be the same as the sum
extended on all polynomials. In fact we may restrict the summation to the $N_j$ of
degree $\le n$, because otherwise there is a factor $w(q^n/q^{\deg(M_r)})=0$ in
the corresponding term. This proves equation \eqref{E:exact}.
\end{proof}

\section{Main part of the asymptotic formula}
We now establish the asymptotic formula in the main theorem.
We begin by establishing the main term and the error term of
$g(n)$. 
The main part of
$w(q^m)$ is $u(q^m):=\frac{q^{m+1}}{q-1}$. We define the error term $v(q^m)$ as
the difference, so that for any integer $m$ we have
\begin{equation}
w(q^m)=u(q^m)+v(q^m).
\end{equation}
We will need the following properties of $u(q^m)$ and $v(q^m)$:
\begin{lem}\label{L:uno}
\begin{itemize}
\item[(a)] For all integers $m\in\Z$ we have $|v(q^m)|\le 1$.
\item[(b)] For all integers $m\in\Z$ we have $|w(q^m)|\le u(q^m)$.
\end{itemize}
\end{lem}

\begin{proof}
For $m<0$ we have $w(q^m)=0$ and therefore $u(q^m)=\frac{q^{m+1}}{q-1}=-v(q^m)$.
Since $m+1\le0$, we have $|v(q^m)|=\frac{q^{m+1}}{q-1}\le \frac{1}{q-1}\le 1$.
Also, in the case $m<0$ we have $|w(q^m)|=0\le u(q^m)$.

For $m\ge0$ we have $w(q^m)=\frac{q^{m+1}-1}{q-1}$, $u(q^m)=\frac{q^{m+1}}{q-1}$ and
therefore $v(q^m)=-\frac{1}{q-1}$. So that $|v(q^m)|=\frac{1}{q-1}\le 1$.  On the other
hand $|w(q^m)|=w(q^m)=\frac{q^{m+1}-1}{q-1}< u(q^m)$.
\end{proof}

To separate the main part in the sum \eqref{E:exact} we will use the following
simple lemma, easily proved by induction.
\begin{lem}\label{L:dos}
Let $w_r=u_r+v_r$ for $1\le r\le R$ be elements of any ring, then \begin{equation} \prod_{r=1}^R w_r=\prod_{r=1}^R u_r+\sum_{s=1}^R \Bigl(v_s\cdot \prod_{j=1}^{s-1}u_j \cdot \prod_{j=s+1}^R w_j\Bigr), \end{equation} where the empty products are equal to $1$.
\end{lem}

We can now establish the following.
\begin{lem}\label{P:start}
The number $g(n)$, of tuples $(Q_1,\dots, Q_v)$  with $\deg (Q_a)\le n$ satisfying
the conditions of coprimality given by the graph $G$, is given by
\begin{equation}\label{E:first}
g(n)=\Bigl(\frac{q^n}{q-1}\Bigr)^{v}\sum_{\substack{(N_1,\dots,N_e)\\\deg(N_a)\le n}}
\frac{\mu(N_1)\cdots\mu(N_e)}{q^{m_1}\cdots q^{m_v}}+\sum_{k=1}^v R_k,
\end{equation}
where $m_r=\deg(M_r)$.

The error terms $R_k$ are bounded by
\begin{equation}\label{E:boundR}
|R_k|\le \Bigl(\frac{q^n}{q-1}\Bigr)^{v-1}
\sum_{\substack{(N_1,\dots,N_e)\\\deg(N_a)\le n}}
\frac{|\mu(N_1)\cdots\mu(N_e)|}{q^{m_1}\cdots \widehat{q^{m_k}}\cdots q^{m_v}},
\end{equation}
where $\widehat{q^{m_k}}$ indicates that this factor is omitted.
\end{lem}

\begin{proof}
Applying Lemma \ref{L:dos} to the exact expression of $g(n)$ in \eqref{E:exact} and using the
decomposition $w(q^m)=u(q^m)+v(q^m)$ we obtain, with $m_r:=\deg M_r$,
\begin{align}\label{E:sep}
g(n)&=\sum_{\substack{(N_1,\ldots,N_e)\\\deg(N_a)\le n}}\mu(N_1)\cdots\mu(N_e)\prod_{r=1}^v u\Bigl(\frac{q^n}{q^{m_{r}}}\Bigr)\notag\\
&+\sum_{\substack{(N_1,\ldots,N_e)\\\deg(N_a)\le n}}\mu(N_1)\cdots\mu(N_e)
v\Bigl(\frac{q^n}{q^{m_1}}\Bigr)\prod_{r=2}^v
w\Bigl(\frac{q^n}{q^{m_r}}\Bigr)\notag\\
&+\sum_{\substack{(N_1,\ldots,N_e)\\\deg(N_a) \le n}}\mu(N_1)\cdots\mu(N_e)
u\Bigl(\frac{q^n}{q^{m_1}}\Bigr)v\Bigl(\frac{q^n}{q^{m_2}}\Bigr)
\prod_{r=3}^v w\Bigl(\frac{q^n}{q^{m_r}}\Bigr)\notag\\
&\cdots\notag\\
&+\sum_{\substack{(N_1,\ldots,N_e)\\\deg(N_a) \le n}}\mu(N_1)\cdots\mu(N_e)
u\Bigl(\frac{q^n}{q^{m_1}}\Bigr)\cdots v\Bigl(\frac{q^n}{q^{m_{v-1}}}\Bigr)
w\Bigl(\frac{q^n}{q^{m_v}}\Bigr)\notag\\
&+\sum_{\substack{(N_1,\ldots,N_e)\\\deg(N_a) \le n}}\mu(N_1)\cdots\mu(N_e)
u\Bigl(\frac{q^n}{q^{m_1}}\Bigr)\cdots u\Bigl(\frac{q^n}{q^{m_{v-1}}}\Bigr)
v\Bigl(\frac{q^n}{q^{m_v}}\Bigr)\notag\\
&=\sum_{\substack{(N_1,\ldots,N_e)\\\deg(N_a) \le n}}\mu(N_1)\cdots\mu(N_e)\prod_{r=1}^v u\Bigl(\frac{q^n}{q^{m_r}}\Bigr)+
\sum_{k=1}^vR_k,
\end{align}
where for $1 \le k \le v$,
\begin{multline*}
R_k=\sum_{\substack{(N_1,\ldots,N_e)\\\deg(N_a) \le n}}\mu(N_1)\cdots\mu(N_e)
\cdot \\
u\Bigl(\frac{q^n}{q^{m_1}}\Bigr)\cdots u\Bigl(\frac{q^n}{q^{m_{k-1}}} \Bigr)
v\Bigl(\frac{q^n}{q^{m_{k}}}\Bigr)
w\Bigl(\frac{q^n}{q^{m_{k+1}}}\Bigr)\cdots
w\Bigl(\frac{q^n}{q^{m_v}}\Bigr).
\end{multline*}

Since $u(q^m)=\frac{q^m}{q-1}$  the main term can be written as
\begin{multline*}
\sum_{\substack{(N_1,\ldots,N_e)\\\deg(N_a) \le n}}\mu(N_1)\cdots\mu(N_e)\prod_{r=1}^v u\Bigl(\frac{q^n}{q^{m_r}}\Bigr)\\
=\Bigl(\frac{q^n}{q-1}\Bigr)^v
\sum_{\substack{(N_1,\dots,N_e)\\\deg(N_a)\le n}}
\frac{\mu(N_1)\cdots\mu(N_e)}{q^{m_1+\cdots+m_v}}
\end{multline*}
On the other hand, thanks to Lemma \ref{L:uno}, the error term can be bounded
by
$$|R_k|\le \sum_{\substack{(N_1,\dots,N_e)\\\deg(N_a)\le n}}
|\mu(N_1)\cdots\mu(N_e)|\prod_{\substack{1\le r\le v\\r\ne k}}
\Bigl|u\Bigl(\frac{q^n}{q^{m_r}}\Bigr)\Bigr|$$
The factor $k$ is missing because $|v(q^m)|\le1$. Therefore we obtain
$$|R_k|\le \Bigl(\frac{q^n}{q-1}\Bigr)^{v-1}
\sum_{\substack{(N_1,\dots,N_e)\\\deg(N_a)\le n}}
\frac{|\mu(N_1)\cdots\mu(N_e)|}{q^{m_1}\cdots \widehat{q^{m_k}}\cdots q^{m_v}},$$
which completes the proof.
\end{proof}

\section{The coefficient of the main term}
We consider now the coefficient
$$\sum_{\substack{(N_1,\dots,N_e)\\\deg(N_a)\le n}}
\frac{\mu(N_1)\cdots\mu(N_e)}{q^{m_1}\cdots q^{m_v}}$$
of our first expression \eqref{E:first} for $g(n)$.  It is clear that the number
of terms added increases with $n$.  We will see that the sum has a limit when $n\to\infty$.
We require some results regarding certain multiplicative functions.

\subsection{Multiplicative functions in Dedekind domains}
In this section we assume, given a Dedekind domain $\cD$ with the finite norm property,
that $I(\cD)$ is the set of non-null ideals in $\cD$ and $\cR(\cD)$ is the ring of ideals
defined in Section \ref{SS:RD}.  Also $G=(V,E)$ is a given finite graph as in
Section \ref{S:PS}.

\begin{definition}
A function $f\colon I(\cD)\to\C$ defined on the set of non-null ideals, is
\emph{multiplicative} if for any relatively prime pair of ideals $Q$, $R$ we
have $f(QR)=f(Q)f(R)$.
\end{definition}

One important example is the M\" obius function $\mu(Q)=(-1)^{\omega(Q)}$ when
$Q$ is the product of
$\omega(Q)$  different primes ideals, and $\mu(Q)=0$ if there is a prime $P$
with $P^2\mid Q$.  An ideal $Q$ with $|\mu(Q)|=1$ is called \emph{squarefree}.
It is well known that
\begin{lem}[Euler product]\label{P:Euler}
Let $f\colon I(\cD)\to\C$ be a multiplicative function. We have the identity
\begin{equation}
\sum_{Q}\frac{f(Q)}{Q}=\prod_P\Bigl(1+\frac{f(P)}{P}+\frac{f(P^2)}{P^2}+\cdots\Bigr).
\end{equation}
\end{lem}
We recall that we are considering in $\cR(\cD)$ the product topology
of $\C^{I(\cD)}$ giving to $\C$ the discrete topology.  The important thing for us is that this implies
the equality
\begin{equation}
\sum_{Q}\frac{f(Q)}{\cN(Q)}=\prod_P\Bigl(1+\frac{f(P)}{\cN(P)}+\frac{f(P^2)}{\cN(P)^2}+\cdots\Bigr),
\end{equation}
when one of the two sides of the equation converges absolutely.

\subsection{Multiplicative functions associated to a graph}

\begin{lem}\label{P:Mul}
Let $f\colon I(\cD)\to\C$ be a multiplicative function,  $G$ a graph and $r$ a
vertex in $G$. Then the two functions $g_{G,f}$ and $g_{G,f}^r\colon I(\cD)\to\C$,
defined by
\begin{align}
g_{G,f}(Q)&=\sum_{\substack{(N_1,\dots, N_e)\\ M_1\cdots M_v=Q}}f(N_1)\cdots f(N_e)
\\
and\\
g_{G,f}^r(Q)&=\sum_{\substack{(N_1,\dots, N_e)\\M_1\cdots  \widehat{M_r}\cdots M_v=Q}}f(N_1)\cdots f(N_e),
\end{align}
are multiplicative. 
\end{lem}

\begin{proof}
Consider the second function (the other is analogous and simpler).
First we show that the sum defining $g_{G,f}^r(Q)$ is finite. Let $(N_1,\dots, N_e)$ be
an edge labeling such that $M_1\cdots  \widehat{M_r}\cdots M_v=Q$.
Any edge
$a$ contains a vertex $s\ne r$, therefore $N_a\mid M_s$ and therefore $N_a\mid Q$.
Therefore the sum can be restricted to edge labelings formed with divisors of $Q$. These
are of finite number.  Therefore the two functions are well defined.

Consider now an edge labelling $(N_1,\dots, N_e)$ such that
$M_1\cdots  \widehat{M_r}\cdots M_v=Q_1Q_2$ with $\gcd(Q_1,Q_2)=1$.
By the previous reasoning we have
$N_a\mid Q_1Q_2$ for any edge $a$. Therefore  we can find another two
edge labelings $(N_{1,1},\dots, N_{1,e})$ and $(N_{2,1},\dots, N_{2,e})$ such that
$N_{1,a}\mid Q_{1}$, $N_{2,a}\mid Q_2$ and $N_{1,a}N_{2,a}=N_a$, for any edge $a$.  It is easy to
see that in this case (with $\gcd(Q_1,Q_2)=1$) we have $M_s=M_{1,s}M_{2,s}$
for any vertex $s$, and
$$M_{1,1}\cdots  \widehat{M_{1,r}}\cdots M_{1,v}=Q_1,\quad
M_{2,1}\cdots  \widehat{M_{2,r}}\cdots M_{2,v}=Q_2.$$
Analogously if we start with two edge labelings $(N_{i,1},\dots, N_{i,e})$ for $i=1$, $2$,  satisfying the above relations, the edge labeling
formed with $N_a=N_{1,a}N_{2,1}$ will satisfy $M_1\cdots  \widehat{M_r}\cdots M_v=Q_1Q_2$.  Notice also that since $\gcd(N_{1,a},N_{2,a})=1$ we have
$f(N_a)=f(N_{1,a})f(N_{2,a})$.

Therefore
\begin{multline*}
g_{G,f}^r(Q_1Q_2)=\sum_{\substack{(N_1,\dots, N_e)\\M_1\cdots  \widehat{M_r}\cdots M_v=Q}}f(N_1)\cdots f(N_e)=\\
\sum_{\substack{(N_{1,1},\dots, N_{1,e})\\M_{1,1}\cdots  \widehat{M_{1,r}}\cdots M_{1,v}=Q_1}}\mskip -10mu f(N_{1,1})\cdots f(N_{1,e})
\sum_{\substack{(N_{2,1},\dots, N_{2,e})\\M_{1,1}\cdots  \widehat{M_{2,r}}\cdots M_{2,v}=Q_2}}\mskip -10mu f(N_{2,1})\cdots f(N_{2,e}).
\end{multline*}
In other words $g_{G,f}^r(Q_1Q_2)=g_{G,f}^r(Q_1)g_{G,f}^r(Q_2)$.\end{proof}

We will need to consider the case of the multiplicative functions
$f=\mu$ or $f=|\mu|$. Therefore we define
\begin{align*}
f_G(Q)&:=\sum_{\substack{(N_1,\dots, N_e)\\M_1\cdots M_v=Q}}\mu(N_1)\cdots \mu(N_e),\\
f_G^+(Q)&:=\sum_{\substack{(N_1,\dots, N_e)\\M_1\cdots M_v=Q}}
|\mu(N_1)\cdots \mu(N_e)|\\
g_{G,r}(Q)&:=
\sum_{\substack{(N_1,\dots, N_e)\\M_1\cdots  \widehat{M_r}\cdots M_v=Q}}\mu(N_1)\cdots \mu(N_e),\\
g_{G,r}^+(Q)&:=\sum_{\substack{(N_1,\dots, N_e)\\M_1\cdots  \widehat{M_r}\cdots M_v=Q}}
|\mu(N_1)\cdots \mu(N_e)|
\end{align*}
By Lemma \ref{P:Mul} these four functions are multiplicative so that
their values are determined by their values in powers $P^n$ of prime ideals $P$.
It is very easy to see that these values $f_G(P^n)$, $f_G^+(P^n)$, $g_G(P^n)$,
$g_G^+(P^n)$ are rational integers independent of the special Dedekind domain, because
the divisors of $P^n$ are $1$, $P$, $P^2$, \dots, $P^n$ in any Dedekind domain.

In \cite{Ari}, using slightly different notation, we considered two polynomials associated to a graph $G$. Namely,
\begin{equation}
Q_G(z)=\sum_{F\subset E}(-1)^{|F|}z^{|v(F)|},\quad
Q_G^+(z)=\sum_{F\subset E}z^{|v(F)|},
\end{equation} where $v(F)=\cup_{\{r,s\}\in\F}\{r,s\}$; the set of all vertices adjacent to edges contained in $F$.
We proved
\begin{lemma}\label{L:poly}
For any graph $G$ and prime $P$ the value $f_G(P^k)$ \emph{(}respectively of $f_G^+(P^k)$
\emph{)} is
equal to the coefficient of $z^k$ in the polynomial $Q_G(z)$, \emph{(}respectively
in the polynomial $Q_G^+(z)$\emph{)}.

In particular we have $f_G(P)=f^+_G(P)=0$.
\end{lemma}

To study the two functions $g_{G,r}$ and $g_{G,r}^+$ we introduce two other
polynomials,
\begin{equation}
Q_{G,r}(z)=\sum_{F\subset E}(-1)^{|F|}z^{|v(F)\smallsetminus\{r\}|},\quad
Q_{G,r}^+(z)=\sum_{F\subset E}z^{|v(F)\smallsetminus\{r\}|}.
\end{equation}

\begin{lemma}\label{L:poly-vertex}
Let  $G$ be a graph  and $r$ one of its vertices.  For any prime $P$ the value $g_{G,r}(P^k)$ \emph{(}respectively $g_{G,r}^+(P^k)$\emph{)} is
equal to the coefficient of $z^k$ in the polynomial $Q_{G,r}(z)$ \emph{(}respectively
in the polynomial $Q_{G,r}^+(z)$\emph{)}.

In particular we have $g_{G,r}^+(P)=-g_{G,r}(P)=d_r$, where $d_r$ is the degree of the vertex
$r$ and $g_{G,r}^+(P^m)=-g_{G,r}(P^m)=0$ for $m\ge v$.
\end{lemma}

\begin{proof}
Consider, for example, the case of $g_{G,r}(P^k)$. By definition
$$g_{G,r}(P^k):=
\sum_{\substack{(N_1,\dots, N_e)\\M_1\cdots  \widehat{M_r}\cdots M_v=P^k}}
\mu(N_1)\cdots \mu(N_e).$$
Any edge labeling $(N_1,\dots, N_e)$ giving a non-null term satisfies $N_j\mid P^k$ for $1\le j\le e$.
Therefore for each $j$ we have $N_j=1$ or $N_j=P$. In this way each non-null term is
associated bijectively  to a subset $F\subset E$: the set of $j$ for which $N_j=P$.
In this case $\mu(N_1)\cdots \mu(N_e)=(-1)^{|F|}$. The elements of the corresponding
vertex labelling $(M_1,\dots M_v)$ satisfies also $M_s=1$ or $M_s=P$. Precisely $M_s=P$ if
$s\in v(F)$.  Since
$M_1\cdots  \widehat{M_r}\cdots M_v=P^k$  we have $k=|v(F)\smallsetminus\{r\}|$.
It follows that $g_{G,r}(P^k)$ is the coefficient of the polynomial $Q_{G,r}(z)$.

We have $|v(F)\smallsetminus\{r\}|=1$ just in the case $F$ consists only of an edge
$e=\{r,s\}$
with an extreme equal to $r$. There are precisely $d_r$ such edges. Therefore
the term of first degree in $Q_{G,r}^+(z)$ is $d_r$. In the case of $Q_{G,r}(z)$
these same terms appear with a factor $(-1)^{|F|}=-1$.

Finally notice that $|v(F)\smallsetminus\{r\}|\le v-1$. Therefore this is the maximum
degree of any term of the polynomials $Q_{G,r}(z)$ and $Q_{G,r}^+(z)$.
\end{proof}

\subsection{A particular multiplicative function}

We define a function that we will use later in the bound of the error terms in
our approximation to $g(n)$.

\begin{definition}\label{D:frs}
Let $G=(V,E)$ be a graph that has two vertices $r\ne s$ with $r$, $s\in V$, that are not joined by an edge. That is, $\{r,s\}\not\in E$
We define a function
$f_{r,s}\colon I(\cD)\to\cR(\cD)$ from the ideals to the ring of ideals by
\begin{equation}
f_{r,s}(Q)=\sum_{\substack{(N_1,\dots, N_e)\\ N_j\mid Q}}\frac{|\mu(N_1)\cdots\mu(N_e)|}
{M_1\cdots\widehat{M_r}\cdots\widehat{M_s}\cdots M_v}.
\end{equation}
We sum on all edge labeling formed with divisors of $Q$ and we omit the two factors
$M_r$ and $M_s$ corresponding to the vertices of the pair $\{r,s\}$.
\end{definition}

\begin{lem} The function $f_{r,s}(Q)$ is multiplicative. That is,
$\gcd(Q_1,Q_2)=1$ implies  $f_{r,s}(Q_1Q_2)=f_{r,s}(Q_1)f_{r,s}(Q_2)$.
\end{lem}

\begin{proof}
If we assume $\gcd(Q_1,Q_2)=1$, then any edge labeling $(N_1,\dots, N_e)$
with $N_j\mid Q_1Q_2$ can be obtained in a unique way from two
edge labelings $(N_{1,1},\dots, N_{1,e})$ and $(N_{2,1},\dots, N_{2,e})$
with $N_{i,j}\mid Q_i$ by the equations
$N_j=N_{1,j}N_{2,j}$. The corresponding vertex labeling then satisfies
$M_j=M_{1,j}M_{2,j}$, and the result follows.
\end{proof}

\begin{lem}
Let $Q_{r,s}(z)$ be the polynomial
\begin{equation}\label{D:Qrs}
Q_{r,s}(z)=\sum_{F\subset E}z^{|v(F)\smallsetminus\{r,s\}|}.
\end{equation}
For any natural number $m$ and prime ideal $P$ we have
\begin{equation}\label{E:spfunc}
f_{r,s}(P^m)=f_{r,s}(P)=Q_{r,s}\Bigl(\frac{1}{P}\Bigr).
\end{equation}
\end{lem}

\begin{proof}
In the definition of $f_{r,s}(P^m)$ we have to sum for each edge labeling
$(N_1,\dots, N_e)$ of  $G$, where $N_j\mid P^m$. Each term of the sum has
a coefficient $|\mu(N_1)\cdots\mu(N_e)|$. Therefore we have only to consider
the term with $N_j=1$ or $N_j=P$.  Any such labeling is
determined by the set  $F=\{j\colon 1\le j\le e, N_j=P\}$. Therefore each
non-null term corresponds to a subset $F\subset E$.  It is clear that for this
labeling the corresponding vertex labeling $(M_1,\dots, M_v)$ will have
$M_t=1$ or $M_t=P$. Precisely the set $v(F)=\cup_{\{t_1,t_2\}\in F}\{t_1,t_2\}$
coincides with the set of vertices $t$ with  $M_t=P$. Therefore the corresponding
term is equal to $1/P^{|v(F)\smallsetminus\{r,s\}|}$.

After this reasoning it is clear that we have \eqref{E:spfunc}.
\end{proof}
We will also need the following.
\begin{lem}\label{C:frspol}
We have
\begin{equation}
f_{r,s}(P)=1+\frac{a_1}{P}+\cdots+\frac{a_{v-2}}{P^{v-2}},
\end{equation}
where
the sum of all coefficients $1+a_1+\cdots+a_{v-2}=2^e$.
\end{lem}

\begin{proof}
The sum of the coefficients is the number of non-null
terms in the sum \eqref{D:Qrs}. Therefore $1+a_1+\cdots+a_{v-2}=2^e$.
\end{proof}

\subsection{Formula for the coefficient of the main term}
We are now able to quantify the coefficient of the main term.
\begin{lem}\label{T:coef}
For any given finite graph $G=(V,E)$, the series
\begin{equation}\label{E:coef}
\sum_{(N_1,\dots,N_e)}
\frac{\mu(N_1)\cdots\mu(N_e)}{q^{m_1}\cdots q^{m_v}}=\prod_PQ_G(q^{-\deg(P)}),
\end{equation}
where the sum extends to all possible edge labelings, and $m_r=\deg(M_r)$
converges absolutely.
\end{lem}

\begin{proof}
First notice that $q^{m_r}=\cN(M_r)$, therefore our sum is the norm of the member
of the ring of ideals
\begin{equation}
Z=\sum_{(N_1,\dots,N_e)}
\frac{\mu(N_1)\cdots\mu(N_e)}{M_1\cdots M_v}.
\end{equation}
Since we want to show absolute convergence we consider instead
\begin{equation}
Z^+=\sum_{(N_1,\dots,N_e)}
\frac{|\mu(N_1)\cdots\mu(N_e)|}{M_1\cdots M_v}.
\end{equation}
In this sum when we take the norm each element is positive, therefore for the
convergence we may reorder terms. We do so by joining the terms for
which $M_1\cdots M_v=Q$ a given ideal. For each $Q$ we are associating a finite
number of terms of the sum (see proof of Lemma \ref{P:Mul}).
We obtain
$$Z^+=\sum_Q\frac{1}{Q}\Bigl(\sum_{\substack{(N_1,\dots, N_e)\\M_1\cdots M_v=Q}}
|\mu(N_1)\cdots\mu(N_e)|\Bigr)=\sum_Q \frac{f_G^+(Q)}{Q}.$$
By Lemma \ref{P:Euler} (Euler product) we have
$$Z^+=\prod_P\Bigl(1+\frac{f_G^+(P)}{P}+\frac{f_G^+(P^2)}{P^2}+\cdots\Bigr).$$
By Lemma \ref{L:poly} we have also $f_G^+(P)=0$.

Taking norms we have
$$\cN(Z^+)=\prod_P\Bigl(1+\frac{f_G^+(P^2)}{q^{2\deg(P)}}+\frac{f_G^+(P^3)}{q^{3\deg(P)}}+\cdots\Bigr).$$
The number of monic irreducible polynomials of degree $n$ is $\le q^n$ (see Lemma
\ref{r(m)}). The sum of
the coefficients of $Q^+_G(z)$ is a constant $C=Q^+_G(1)$. Therefore
$$\cN^+(Z)\le \prod_{n=0}^\infty \Bigl(1+\frac{C}{q^{2n}}\Bigr)^{q^n}$$
It is clear that this is finite because taking logarithms yields
$$\sum_{n=0}^\infty q^n\log\Bigl(1+\frac{C}{q^{2n}}\Bigr)\le C\sum_{n=0}^\infty
\frac{1}{q^n}= C\frac{q}{q-1}<\infty$$

By similar reasoning we obtain
$$Z=\sum_Q\frac{1}{Q}\Bigl(\sum_{\substack{(N_1,\dots, N_e)\\M_1\cdots M_v=Q}}
\mu(N_1)\cdots\mu(N_e)\Bigr)=\sum_Q \frac{f_G(Q)}{Q}.$$
By Lemma \ref{P:Euler} (Euler product) we have
$$Z=\prod_P\Bigl(1+\frac{f_G(P)}{P}+\frac{f_G(P^2)}{P^2}+\cdots\Bigr).$$
Taking norms
$$\cN(Z)=\sum_{(N_1,\dots,N_e)}
\frac{\mu(N_1)\cdots\mu(N_e)}{q^{m_1}\cdots q^{m_v}}
=\prod_PQ_G(q^{-\deg(P)}),$$
where the sum extends to all monic irreducible polynomials $P$ in $\F_q	[z]$.
\end{proof}

Let $\rho_{G,q}$ be the sum in \eqref{E:coef}.
We have now established an interim result; expressing $g(n)$ as a fully developed main term and a number of error terms as follows.
\begin{lem}
The number $g(n)$, of tuples $(Q_1,\dots, Q_v)$  with $\deg (Q_a)\le n$ satisfying
the conditions of coprimality given by the graph $G$, is
\begin{equation}\label{E:second}
g(n)=\rho_{G,q}\Bigl(\frac{q^n}{q-1}\Bigr)^{v}-T+\sum_{k=1}^v R_k
\end{equation}
where
\begin{equation}\label{E:defT}
T= \Bigl(\frac{q^n}{q-1}\Bigr)^{v}\sum_{
\substack{(N_1,\dots,N_e)\\\text{some }\deg(N_a)> n}}
\frac{\mu(N_1)\cdots\mu(N_e)}{q^{m_1}\cdots q^{m_v}},
\end{equation}
and $R_k$ satisfies the bound in \eqref{E:boundR}.
\end{lem}

\section{Bound on the error terms}
In this Section we obtain bounds on the terms $T$ and $R_k$ in \eqref{E:second} and \eqref{E:defT}.

\subsection{Bound on $R_k$}

Let $r_q(n)$ be the number of monic irreducible polynomials in $\F_q[z]$ of degree
$n$. It is well known that (see, for example, \cite[Th.~3.25, p.84]{LiNi})
\begin{equation}\label{irredcount}
r_q(n)=\frac{1}{n}\sum_{d\mid n}\mu\(n/d\)q^{d}.
\end{equation}
We will need the following lemma.
\begin{lemma}
\label{r(m)}
The number $r_q(n)$ of irreducible polynomials of degree $=n$ in $\F_q[z]$ is bounded
by $r_q(n)\le \frac1n q^n$.
\end{lemma}

\begin{proof}
By \eqref{irredcount} we have to prove that $\sum_{d\mid n} \mu(n/d)q^{d}\le q^n$.
If $n=1$ or $n=p$ is a prime number,  this is trivial.
When $n$ is composite
let $p$ be the least prime number dividing $n$. The divisors of $n$ in decreasing order are
$$n, \quad n/p,\quad d_3,\quad d_4,\dots $$
Therefore we have
\begin{align*}
\sum_{d\mid n} \mu\(n/d\)q^{d}&=q^n-q^{n/p}+\sum_{k\ge 3} \mu(n/d_k) q^{d_k}\\
&\le q^n-q^{n/p}+\sum_{d=1}^{n/p-1}q^d=q^n-q^{n/p}+\frac{q^{n/p}-1}{q-1}<q^n.
\end{align*}
\end{proof}
We now state and prove a bound on $R_k$.
\begin{lem}
Let $d$ be the maximum degree of a vertex in the graph $G$. The error terms
$R_k$ are bounded by
\begin{equation}\label{RkB}
\Bigl|\sum_{k=1}^vR_k\Bigr|\le \exp(d) 2^{2^e}v n^d\,\Bigl(\frac{q^n}{q-1}\Bigr)^{v-1}.
\end{equation}
\end{lem}

\begin{proof}
The proof is similar to the proof of Lemma \ref{T:coef}.
We consider the element of the ring of ideals
$$W=\sum_{\substack{(N_1,\dots,N_e)\\\deg(N_a)\le n}}
\frac{|\mu(N_1)\cdots\mu(N_e)|}{M_1\cdots \widehat{M_k}\cdots M_v},$$
whose norm is equal to the sum appearing in \eqref{E:boundR}. That is,
$$\cN(W)=\sum_{\substack{(N_1,\dots,N_e)\\\deg(N_a)\le n}}
\frac{|\mu(N_1)\cdots\mu(N_e)|}{q^{m_1}\cdots \widehat{q^{m_k}}\cdots q^{m_v}}.$$

The primes $P$ that divide the ideals $M_1\cdots \widehat{M_k}\cdots M_v$ appearing
in the denominators of $W$ divide some of the $N_a$. Therefore $\deg(P)\le n$.
It follows that
$$W\vartriangleleft \sideset{}{^*}\sum_{Q}\frac{1}{Q}
\Bigl(\sum_{M_1\cdots \widehat{M_k}\cdots M_v=Q}|\mu(N_1)\cdots\mu(N_e)|\Bigr),$$
where the $*$ in the sum indicates that we restrict the sum to monic polynomials
all of whose prime factors have degree $\le n$. This is not an equality because
there may be some edge labelings $(N_1,\dots, N_e)$ that satisfy
$M_1\cdots \widehat{M_k}\cdots M_v=Q$  that are not contained in $W$ because of  the restriction
on the degrees.

Noticing the definition of $g^+_{G,k}(Q)$ we have
$$\sideset{}{^*}\sum_{Q}\frac{1}{Q}
\Bigl(\sum_{M_1\cdots \widehat{M_k}\cdots M_v=Q}|\mu(N_1)\cdots\mu(N_e)|\Bigr)=
\sideset{}{^*}\sum_{Q}\frac{g^+_{G,k}(Q)}{Q}.$$

Reasoning as in the Euler product proof we obtain
$$W\vartriangleleft\prod_{\deg(P)\le n}\Bigl(1+\frac{g^+_{G,k}(P)}{P}+
\frac{g^+_{G,k}(P^2)}{P^2}+\cdots\Bigr).$$
By Lemma \ref{L:poly-vertex} we have $g^+_{G,k}(P)=d_k$, the degree of the vertex $k$ in $G$.
And $g^+_{G,k}(P^m)=0$ for $m\ge v$.  Let $C=Q^+_{G,k}(1)\le 2^e$ be the sum of $g^+_{G,k}(P^m)$
for $0\le m\le v$. Taking norms in the above relation $\vartriangleleft$, yields
$$\cN(W)\le \prod_{m=1}^n \Bigl(1+\frac{d_k}{q^m}+\frac{C}{q^{2m}}\Bigr)^{r(m)},$$
where $r(m)$ is the number of monic irreducible polynomials of degree $m$.

From Lemma \ref{r(m)} we have $r(m)\le \frac1mq^m$, and so
\begin{align*}\log \cN(W)\le \sum_{m=1}^n \frac{q^m}{m}\log \Bigl(1+\frac{d_k}{q^m}+
\frac{C}{q^{2m}}\Bigr)\le \sum_{m=1}^n \frac{q^m}{m}\Bigl(\frac{d_k}{q^m}+
\frac{C}{q^{2m}}\Bigr)\\
=d_k\sum_{m=1}^n\frac1m+\sum_{m=1}^n \frac{2^e}{mq^m}\le (1+\log n)d_k+2^e \log2.
\end{align*}

Using the bound on $\cN(W)$ in \eqref{E:boundR} yields
$$|R_k|\le 2^{2^e} (\exp(1) n)^{d_k}\Bigl(\frac{q^n}{q-1}\Bigr)^{v-1}.$$
This proves our lemma.
\end{proof}

\subsection{Bound on the error term $T$}
To bound the error term $T$ we require the following lemma.

\begin{lemma}
For any squarefree polynomial $Q\in \F_q[z]$ of degree $n>1$ we have
\begin{equation}\label{E:ogbn}
\omega(Q)\le 4\frac{n}{\log n}\log q.
\end{equation}
\end{lemma}

\begin{proof}
Let $x=\frac12\frac{\log n}{\log q}$. We may assume $x\ge1$, for otherwise
$$\omega(Q)\le n\le \frac{n}{x}=2\frac{n}{\log n}\log q<4\frac{n}{\log n}\log q.$$

So we assume $x\ge 1$. Since  $Q$ is squarefree we have
$Q=Q_1 Q_2$ where $Q_1$ is the product of
irreducible polynomials $P\mid Q$ with degree  $\le x$ and $Q_2$ is the product of
irreducible polynomials $P\mid Q$ with degree $> x$.

The polynomial $Q_1$ is a divisor of $R_1$, the product of all irreducibles of degree
$\le x$. That is,
$$R_1=\prod_{k\le x}\prod_{\deg(P)=k}P.$$
We have
$$\omega(R_1)=\sum_{k\le x} r_q(k)\quad\text{and}\quad
\deg(R_1)=\sum_{k\le x} k r_q(k).$$
All irreducible factors of $Q_2$ have degree $> x$. Therefore
$$x\omega(Q_2)< \deg(Q_2)=n-\deg(Q_1).$$
Therefore
$$\omega(Q)=\omega(Q_1)+\omega(Q_2)\le \omega(R_1)+\frac{n-\deg(Q_1)}{x}
\le \frac{n}{x}+ \sum_{k\le x} r_q(k),$$
where $r_q(n)$ is the number of monic irreducible polynomials in $\F_q[z]$ of degree
$n$. By Lemma \ref{r(m)}  we have $r_q(k)\le \frac1k q^k$, so that for $x\ge1$ we obtain
$$\sum_{k\le x} r_q(k)=\sum_{k\le x} \frac{1}{k}q^k\le q^x(1+\log x).$$

Since $x=\frac12\frac{\log n}{\log q}$
\begin{align*}
\omega(Q)&\le 2\frac{n}{\log n}\log q+(1+\log x)\exp\Bigl(\frac12\frac{\log n}{\log q}\log q\Bigr)\\&=2\frac{n}{\log n}\log q+n^{\frac12}\Bigl(1+\log\Bigl(\frac12\frac{\log n}{\log q}\Bigr)\Bigr)\le 4\frac{n}{\log n}\log q.
\end{align*}
Note that
$$n^{\frac12}\Bigl(1+\log\Bigl(\frac12\frac{\log n}{\log q}\Bigr)\Bigr)\le
n^{\frac12}\Bigl(1+\log\Bigl(\frac12\frac{\log n}{\log 2}\Bigr)\Bigr),$$
and
$$\frac{n}{\log n}\log 2\le 2\frac{n}{\log n}\log q.$$
So it will suffice to show that
$$n^{\frac12}\Bigl(1+\log\Bigl(\frac12\frac{\log n}{\log 2}\Bigr)\Bigr)\le 2\frac{n}{\log n}\log 2.$$
This is equivalent to
$$y(1+\log y)\le 2^y, \qquad \text{for}\quad y=\frac12\frac{\log n}{\log 2}\ge x\ge1,$$
which is easily shown to be true.
\end{proof}

We now prove the following bound on $T$.

\begin{lem}
Given any $0<\varepsilon<\frac12$
the term $T$ defined in \eqref{E:defT} is bounded by
\begin{equation}\label{E:TB}
|T|\le 3\sum_{j=1}^e\rho_{G_j',q}^+ \Bigl(\frac{q^n}{q-1}\Bigr)^{v}  q^{-(1-\varepsilon)n},\qquad n>n_0,
\end{equation}
where $G'_j$ is the graph obtained from $G$ after removing the edge $e_j$, and $n_0$ depends on
$\varepsilon$ and the number of vertices in the graph $G$.
\end{lem}

\begin{proof}
By the definition \eqref{E:defT} of $T$ we have

\begin{align*}
|T|&\le \Bigl(\frac{q^n}{q-1}\Bigr)^{v}\sum_{
\substack{(N_1,\dots,N_e)\\\text{some }\deg(N_a)> n}}
\frac{|\mu(N_1)\cdots\mu(N_e)|}{q^{m_1}\cdots q^{m_v}}\\
&\le\Bigl(\frac{q^n}{q-1}\Bigr)^{v}\sum_{a=1}^e
\sum_{
\substack{(N_1,\dots,N_e)\\\deg(N_a)> n}}
\frac{|\mu(N_1)\cdots\mu(N_e)|}{q^{m_1}\cdots q^{m_v}}.
\end{align*}
This is only an inequality because some edge labelings $(N_1,\dots,N_e)$
may have more than one $N_a$ of degree $>n$.
We have to bound each of the $e$ sums. They are all equivalent, in fact
the index of the edges is arbitrary.  Therefore we only bound the one
with $\deg(N_1)>n$. This simplifies our notations a little.

Therefore we bound the sum
$$S:=\sum_{
\substack{(N_1,\dots,N_e)\\\deg(N_1)> n}}
\frac{|\mu(N_1)\cdots\mu(N_e)|}{q^{m_1}\cdots q^{m_v}},$$
and the corresponding element of the ring of ideals
$$W:=\sum_{
\substack{(N_1,\dots,N_e)\\\deg(N_1)> n}}
\frac{|\mu(N_1)\cdots\mu(N_e)|}{M_1\cdots M_v},$$
with $\cN(W)=S$.

The edge $e_1=\{r,s\}$ plays a special role in $W$. We treat this edge
differently from the others. Let $(N_1,\dots,N_e)$ be an edge labeling with
$\deg(N_1)>n$ and  squarefree $N_a$ so that the corresponding term in $W$
is not null.
By definition of the associated labeling (Definition \ref{D:assoc}),
$$M_r=\lcm(N_1,N_{\alpha_1},\dots,N_{\alpha_k}),\quad
M_s=\lcm(N_1,N_{\beta_1},\dots,N_{\beta_\ell}),$$
where we may have $k=0$ or $\ell=0$.

For any other edge $2\le j\le e$ we define $D_j=\gcd(N_1, N_j)$ and then
\begin{equation}\label{E:detail}
N_j=D_jN'_j,\quad D_j\mid N_1.
\end{equation}
Since we assume that $N_j$ is squarefree we have $\gcd(N_1,N'_j)=1$. It is
clear that
\begin{align*}
M_r&=\lcm(N_1,D_{\alpha_1}N'_{\alpha_1},\dots, D_{\alpha_k}N'_{\alpha_k})
=N_1\lcm(N'_{\alpha_1},\dots,N'_{\alpha_k}),\\
M_s&=N_1\lcm(N'_{\beta_1},\dots,N'_{\beta_\ell}).
\end{align*}
For any other vertex $t\ne r$ and $t\ne s$ we have
\begin{align*}
M_t&=\lcm(N_{t_1},\dots, N_{t_m})=\lcm(D_{t_1}N'_{t_1},\dots,D_{t_m}N'_{t_m})\\
&=
\lcm(D_{t_1},\dots,D_{t_m})\lcm(N'_{t_1},\dots,N'_{t_m}),
\end{align*}
where $m$ will depend on $t$.

It follows that the given term of $W$ satisfies
\begin{multline*}
\frac{|\mu(N_1)\cdots\mu(N_e)|}{M_1\cdots M_v}\vartriangleleft
\frac{\mu(N_1)}{N_1^2}
\frac{|\mu(N_2)\cdots \mu(N_e)|}{\lcm(N'_{\alpha_1},\ldots,N'_{\alpha_k})\lcm(N'_{\beta_1},\ldots,N'_{\beta_l})}
\\ \cdot\prod_{\substack{1 \le t \le v\\t\ne r,~t \ne s}} \frac{1}{\lcm(D_{t_1},\ldots,D_{t_m})\lcm(N'_{t_1},\ldots,N'_{t_m})}.
\end{multline*}

This is true even if the $N_j$ are not squarefree because in this case both members
are equal to $0$.

We have
\begin{multline*}
|\mu(N_2)\cdots \mu(N_e)|=|\mu(D_2N'_2)\cdots \mu(D_eN'_e)|\\\le
|\mu(D_2)\cdots\mu(D_e)|\times |\mu(N'_2)\cdots \mu(N'_e)|.
\end{multline*}
It follows that
\begin{multline*}
\frac{|\mu(N_1)\cdots \mu(N_e)|}{M_1\cdots M_v}
\vartriangleleft
\frac{|\mu(N_1)|}{N_1^2}|\mu(D_2)\cdots\mu(D_e)|
\prod_{\substack{1 \le t \le v\\t\ne r,~t \ne s}} \frac{1}{\lcm(D_{t_1},\ldots,D_{t_m})}
\\\cdot
\frac{|\mu(N'_2)\cdots \mu(N'_e)|}{\lcm(N'_{\alpha_1},\ldots,N'_{\alpha_k})\lcm(N'_{\beta_1},\ldots,N'_{\beta_l})} \prod_{\substack{1 \le t \le v\\t\ne r,~t \ne s}} \frac{1}{\lcm(N'_{t_1},\ldots,N'_{t_m})}
\end{multline*}
Consider the graph $G'=(V,E')$ obtained from $G$ by removing the edge $e_1=\{r,s\}$.
Then $(N'_2,\dots, N'_e)$ is an edge labeling for $G'$ and
$M'_t=\lcm(N'_{t_1},\ldots,N'_{t_m})$ for any vertex $t\notin\{r,s\}$ of $G'$, while
$M'_r=\lcm(N'_{\alpha_1},\ldots,N'_{\alpha_k})$ and $M'_s=
\lcm(N'_{\beta_1},\ldots,N'_{\beta_l})$. The above relation can be written as
\begin{multline*}
\frac{|\mu(N_1)\cdots \mu(N_e)|}{M_1\cdots M_v}
\vartriangleleft
\frac{|\mu(N_1)|}{N_1^2}|\mu(D_2)\cdots\mu(D_e)|
\prod_{\substack{1 \le t \le v\\t\ne r,~t \ne s}} \frac{1}{\lcm(D_{t_1},\ldots,D_{t_m})}
\cdot\\\cdot
\frac{|\mu(N'_2)\cdots \mu(N'_e)|}{M'_1\cdots M'_v}.
\end{multline*}
Recall that $D_j$ are divisors of $N_1$. The above implies that
$$W\vartriangleleft W_1 W_2 ,$$
where we define
\begin{align*}
W_1&=\sum_{\deg(N_1)>n}\frac{|\mu(N_1)|}{N_1^2}\sum_{\substack{(D_2,\dots,D_e)\\D_j\mid N_1}}|\mu(D_2)\cdots\mu(D_e)|
\prod_{\substack{1 \le t \le v\\t\ne r,~t \ne s}} \frac{1}{\lcm(D_{t_1},\ldots,D_{t_m})},\\
W_2&=\sum_{(N'_2,\dots, N'_e)}\frac{|\mu(N'_2)\cdots \mu(N'_e)|}{M'_1\cdots M'_v}.
\end{align*}
From the data of $N_1$, $(D_2,\dots,D_e)$ and $(N'_2,\dots, N'_e)$ we
reconstruct uniquely the edge labeling $(N_1,\dots, N_e)$ by the equations
\eqref{E:detail}.
Lemma \ref{T:coef} applied to the graph $G'$ gives us that  $\cN(W_2)=\rho_{G',q}^+$
is a finite constant.

Therefore we need to bound  the norm $\cN(W_1)$ because $W\vartriangleleft W_1W_2$
implies
$$\cN(W)\le \cN(W_1)\cN(W_2).$$
Let $G'$ be the graph $G$ with the edge $\{r,s\}$ removed. Then $(D_2,\dots, D_e)$ is
a edge labeling of $G'$ and $\lcm(D_{t_1},\ldots,D_{t_m})$  are the polynomials
of the corresponding vertex labeling. Therefore, with the notations of Definition
\ref{D:frs}, we have
$$W_1=\sum_{\deg(N)>n}\frac{|\mu(N)|}{N^2}f_{r,s}(N),$$
where $f_{r,s}$ is the function associated to the graph $G'$ and the pair of
vertices, not forming an edge in $G'$, $\{r,s\}$. By Lemma \ref{C:frspol}, for each
irreducible polynomial $P$,
the norm of $f_{r,s}(P)$ is less than or equal $2^{e-1}$ because $e-1$ is the number
of edges of the graph $G'$. Hence if $N$ is squarefree with $\omega(N)$ irreducible
factors, we have $\cN(f_{rs}(N))\le 2^{e\,\omega(N)}$.
It follows that the norm of $W_1$ is less than or equal to
$$\cN\(\sum_{\deg(N)>n}\frac{\mu(N)}{N^2}2^{e\,\omega(N)}\).$$
It will now suffice to calculate a suitable upper bound on this norm.

In fact we will show the following.
\begin{lem}\label{P:thebound}
Given $0<\varepsilon<1/2$ and a natural number $a$ there is an $n_0=n_0(\varepsilon, a)$ such that
\begin{equation}\label{E:thebound}
\cN\Bigl(\sum_{\deg(Q)>n}\frac{|\mu(Q)|2^{a\omega(Q)}}{Q^2}\Bigr)\le 3q^{-(1-\varepsilon)n},
\qquad  n\ge n_0.
\end{equation}
\end{lem}

\begin{proof}
Joining the terms with $\deg(Q)=m$ and applying  \eqref{E:ogbn},  the norm $\cN$ in \eqref{E:thebound} is bounded by
$$\cN\le \sum_{m=n+1}^\infty q^m \frac{1}{q^{2m}}
\exp\Bigl(a(\log 2) 4\frac{m}{\log m}\log q\Bigr).$$

Taking $n_0=n_0(\varepsilon,a)$ we  will have
$$4\frac{a\log 2}{\log m}<\varepsilon,\qquad \text{ for } m\ge n_0.$$
Therefore we have  for $n\ge n_0$,
$$\cN\le \sum_{m=n+1}^\infty  \frac{1}{q^{m}} q^{\varepsilon m}=
\frac{1}{q^{(1-\varepsilon)n}}\frac{1}{q^{1-\varepsilon}-1}.$$
Since $q\ge2$ and $\varepsilon<1/2$, the term $\frac{1}{q^{1-\varepsilon}-1}\le \frac{1}{2^{1/2}-1}<3$ is
bounded by an absolute constant. Thus,
$$\cN\Bigl(\sum_{\deg(Q)>n}\frac{|\mu(Q)|2^{a\omega(Q)}}{Q^2}\Bigr)\le 3q^{-(1-\varepsilon)n},$$
which concludes the proof of Lemma \ref{P:thebound}.
\end{proof}
This concludes the proof of the upper bound on $T$.
\end{proof}
\section{Proof of the main theorem}
We can now prove the main theorem.
\begin{proof}
In \eqref{E:second} we obtained
$$g(n)=\frac{\rho_{G,q}}{(q-1)^v}q^{nv}-T+\sum_{k=1}^v R_k.$$
In \eqref{RkB} we showed that we have
$$\Bigl|\sum_{k=1}^vR_k\Bigr|\le \exp(d) 2^{2^e}v n^d\,\Bigl(\frac{q^n}{q-1}\Bigr)^{v-1},$$
where $d$ is the maximum degree of the vertices of $G$.
And in \eqref{E:TB} we have shown that given $0<\varepsilon<\frac12$ we have
for $n\ge n_0(\varepsilon, e)$
$$|T|\le 3\sum_{j=1}^e\rho_{G_j',q}^+ \Bigl(\frac{q^n}{q-1}\Bigr)^{v}  q^{-(1-\varepsilon)n}.$$
Hence
$$g(n)=\frac{\rho_{G,q}}{(q-1)^v}q^{nv}\Bigl(1-\frac{(q-1)^v}{q^{nv}\rho_{G,q}}T+
\frac{(q-1)^v}{q^{nv}\rho_{G,q}}\sum_{k=1}^v R_k\Bigr).$$
Since
$$\Bigl|\frac{(q-1)^v}{q^{nv}\rho_{G,q}}T\Bigr|\le \frac{3}{\rho_{G,q}}\Bigl(\sum_{j=1}^e \rho_{G'_j,q}^+\Bigr) q^{-(1-\varepsilon)n}, \qquad n\ge n_0(\varepsilon,e)$$
and
$$\Bigl|\frac{(q-1)^v}{q^{nv}\rho_{G,q}}\sum_{k=1}^v R_k\Bigr|\le \frac{\exp(d) 2^{2^e}(q-1)v}{\rho_{G,q}}n^dq^{-n},$$
we obtain
$$g(n)=\frac{\rho_{G,q}}{(q-1)^v}q^{nv}L,
$$
where
$$L=1+\frac{\exp(d) 2^{2^e}(q-1)v}{\rho_{G,q}} O(n^d q^{-n})
+\frac{3}{\rho_{G,q}}\Bigl(\sum_{j=1}^e \rho_{G'_j,q}^+\Bigr)O(q^{-(1-\varepsilon)n}),
$$
and the constants in the $O$ symbols are absolute in both cases.

We can also write this in the simplified form
$$g(n)=\frac{\rho_{G,q}}{(q-1)^v}q^{nv}\Bigl(1+ O_{G,q}(n^d q^{-n})
+O_{G,q}(q^{-(1-\varepsilon)n})\Bigr).
$$
\end{proof}

\section{Acknowledgment}
The authors thanks Igor Shparlinski for pointing out that \cite{Ari} could be adapted for tuples of monic polynomials in finite fields.
\makeatletter
\renewcommand{\@biblabel}[1]{[#1]\hfill}
\makeatother

\end{document}